\documentclass[11pt]{article}

\usepackage{cancel,amsthm,amsmath,times,amsopn,amsfonts,amssymb,epsfig,anysize}
\usepackage{graphicx}
\usepackage[all]{xy}
\usepackage{pifont,txfonts,graphicx,geometry}
\usepackage[square,numbers]{natbib}
\usepackage{pb-diagram}
\usepackage{mathtools}

\usepackage{latexsym,graphics,color}
\usepackage{tabmac_avi}
\usepackage[usenames,dvipsnames]{xcolor}
\usepackage{algorithmic}

\usepackage{tikz}
\usepackage{tkz-euclide}
\usepackage{array}

\newtheorem{theorem}{Theorem}

\newtheorem{corollary}[theorem]{Corollary}
\newtheorem{lemma}[theorem]{Lemma}

\newtheorem{proposition}[theorem]{Proposition}
\newtheorem{definition}[theorem]{Definition}
\newtheorem{remark}[theorem]{Remark}
\newtheorem{example}[theorem]{Example}

\def \kcharge {k\text{-charge}}
\def \kcocharge {k\text{-cocharge}}

\def\charge{ {\rm {charge}}}

\def\cocharge{ {\rm {cocharge}}}
\def\shape{ {\rm {shape}}}

\def \Tleqi{T_{\leq i}}

\def \kcharge {k\text{-charge}}
\def \kcocharge {k\text{-cocharge}}
\def \ktableau {k\text{-tableau}}
\def \ktableaux {k\text{-tableaux}}

\def \iup {i^{\uparrow}}
\def \idown {i^{\downarrow}}
\def \imdown {{(i-1)}^{\downarrow}}

\begin{document} 


\title{Positivity of affine charge}
\author{Avinash J. Dalal 
}

\newcommand{\Addresses}{{
  \bigskip
  \footnotesize

  A.~Dalal, \textsc{Department of Mathematics \& Statistics, University of West Florida,
    Pensacola, FL 32514.}\par\nopagebreak
  \textit{E-mail:} \texttt{adalal@uwf.edu}


}}

\date{}
\maketitle

\begin{abstract} 

The branching of $k-1$-Schur functions into $k$-Schur functions was given by Lapointe, Lam, Morse and Shimozono as chains in a poset on $k$-shapes.  The $k$-Schur functions are the parameterless case of a more general family of symmetric functions over $\mathbb Q(t)$, conjectured to satisfy a $k$-branching formula given by weights on the $k$-shape poset.  A concept of a (co)charge on a $k$-tableau was defined by Lapointe and Pinto.  Although it is not manifestly positive, they prove it is compatible with the $k$-shape poset for standard $k$-tableau and the positivity follows.  Morse introduced a manifestly positive notion of affine (co)charge on $k$-tableaux and conjectured that it matches the statistic of Lapointe-Pinto.  Here we prove her conjecture and the positivity of $k$-(co)charge for semi-standard tableaux follows.

\end{abstract}

\section{Introduction}

The Macdonald basis for the space of symmetric functions is at the center of topics such as double affine Hecke algebras, quantum relativistic systems, diagonal harmonics and Hilbert schemes on points in the plane.  The study was initiated by the work of Macdonald when he conjectured \cite{[M2]} non-negativity of the $q,t$-polynomial coefficients in terms of a shifted basis of Schur functions.  Garsia rephrased his conjecture using a modification of Macdonald's polynomials, $H_\mu(x;q,t)$, as

\begin{equation}
\label{macpolyintro}
H_{\mu}(x;q,t) = \sum_{\lambda}K_{\lambda \mu}(q,t) s_{\lambda}(x)\,, \emph{ where } K_{\lambda \mu}(q,t) \in \mathbb{N}[q,t]\,.
\end{equation}
As such, the $q,t$-Kostka polynomials gained representation theoretic significance and a combinatorial formula has long been sought.

When $q=0$, $K_{\lambda \mu}(0,t)$ are the Kostka-Foulkes polynomials.  These polynomials appear in connection with Hall-Littlewood polynomials \cite{Green}, affine Kazhdan-Lusztig theory \cite{Lu}, affine tensor product multiplicities~\cite{NY:1997}, and they also encode the dimensions of bigraded $S_n$-modules \cite{GP}.  In \cite{LSfoulkes}, Lascoux and Sch\"utzenberger combinatorially characterized these polynomials by associating a non-negative integer statistic called $\charge$ to each semi-standard Young tableau and proved that

\begin{equation}
\label{chargeintro}
K_{\lambda\mu}(0,t) = \sum_{T\in SSYT(\lambda,\mu)} t^{\charge(T)}\,,
\end{equation}
where $SSYT(\lambda, \mu)$ is the set of semi-standard Young tableaux of shape $\lambda$ and weight $\mu$.  

While studying the Macdonald polynomials, a new family of symmetric functions, $s_{\mu}^{(k)}(x;t)$, called $k$-atoms was introduced in~\cite{[LLM]}.  Empirical evidence suggested that, for any positive integer $k$, these functions are indexed by a partition $\mu$ whose parts are not larger than $k$ and 
$$
s_{\mu}^{(k)}(x;t) = \sum_{T \in \mathcal{A}_{\mu}^{(k)}} t^{charge(T)} s_{shape(T)}(x)\,,
$$
where $\mathcal{A}_{\mu}^{(k)}$ is a certain set of tableaux of weight $\mu$.  One striking observation they made was that for any partition $\lambda$ whose parts are not larger than $k$,
$$
H_{\lambda}(x;q,t) = \sum_{\stackrel{ \mu}{\mu_1 \leq k}} K_{\mu \lambda}^{(k)}(q,t) s_{\mu}^{(k)}(x;t)\,, \emph{ where }K_{\mu \lambda}^{(k)}(q,t) \in \mathbb{N}[q,t]\,.
$$
Another observed feature of $k$-atoms is that for large values of $k$, $s_{\mu}^{(k)}(x;t) = s_{\mu}(x;t)$, and thus $K_{\mu \lambda}^{(k)}(x;t)$ reduces to the $q,t$-Kostka polynomials in \eqref{macpolyintro}.  This observation inspired the conjecture in \cite{[LLM]} that the $k$-atoms expand positively in terms of $(k+1)$-atoms:
\begin{equation}
\label{kbranchconj}
s_{\lambda}^{(k)}(x;t) = \sum_{\mu}b_{\lambda \mu}^{(k \rightarrow k+1)}(t)s_{\mu}(x;t)\,,
\end{equation}
where the $k$-branching coefficients, $b_{\lambda \mu}^{(k \rightarrow k+1)}(t) \in \mathbb{N}[t]$.  In~\cite{[LLMS2]}, a poset on partitions called $k$-shapes was introduced and it was conjectured that the $k$-branching coefficients enumerate maximal chains in this poset modulo an equivalence.  The result was proven therein for $t=1$.

A massive effort towards the generic $t$ case was put forth by Lapointe and Pinto~\cite{MR3115329}.  A key focus in their work is the introduction of a statistic on the set of objects ($\ktableaux$) whose enumeration is $K_{\lambda\mu}^{(k)}(1,1)$.  The statistic called $\kcharge$ is conjectured to give, for partition $\lambda$ with $\lambda_1\leq k$,
\begin{equation}
\label{HLitokSchur}
H_{\lambda}(x;t) = \sum_{T}t^{\kcharge(T)}s_{\shape(T)}^{(k)}(x;t)\,,
\end{equation}
where $T$ is a $\ktableau$ of weight $\lambda$.  As evidence to support their definition of $k$-charge, \cite{MR3115329} prove that the $k$-charge for a standard $k$-tableau is compatible with the weak bijection introduced in \cite{[LLMS2]}.

\begin{theorem}
\cite{MR3115329}
\label{LPtheorem}
The weak bijection in the standard case
$$
\text{SWTab}_{\lambda}^k \longrightarrow \bigsqcup_{\mu \in \mathcal{C}^k} \text{SWTab}_{\mu}^{k-1} \times \overline{\mathcal{P}}^k(\lambda,\mu)
$$
is given by $T^{(k)} \longmapsto (T^{(k-1)},[\bf{p}])$, where $T^{(k)}$ is a $\ktableau$, $\text{SWTab}_{\lambda}^k$ is the set of all standard $\ktableau$ of shape $\lambda$ and $[\bf{p}]$ is a certain equivalence class of paths in the $k$-shapes poset, is such that
$$
\kcharge(T^{(k)}) = k'\text{-charge}(T^{(k-1)}) + charge([\bf{p}])\,,
$$
where $k' = k-1$ and $charge([\bf{p}])$ is the charge of the path $\bf{p}$.
\end{theorem}

A consequence of the weak bijection is that the charge of a standard tableau on $n$ letters is the sum of the charge on the corresponding paths in the $k$-shapes poset, for $k=2,3,\ldots,n$.  Namely, iterating the weak bijection starting with a standard tableau $T$ on $n$ letters gives
$$
T \longmapsto (T^{(n-1)},[{\bf{p}}_n]), \hspace{0.2in} T^{(n-1)} \longmapsto (T^{(n-2)},[{\bf{p}}_{n-1}]), \hspace{0.1in} \ldots, \hspace{0.1in} T^{(2)} \longmapsto (T^{(1)},[{\bf{p}}_2])\,,
$$
which puts $T$ in correspondence with $(T^{(1)},[{\bf{p}}_n],[{\bf{p}}_{n-1}],\ldots,[{\bf{p}}_2])$.  Since there is a unique $1$-tableau $T^{(1)}$, then $T$ is in correspondence with the equivalence of paths $([{\bf{p}}_n],[{\bf{p}}_{n-1}],\ldots,[{\bf{p}}_2])$.  Finally, the $\kcharge$ of $T^{(1)}$ being $0$, and the compatibility between the $\kcharge$ and the weak bijection in the standard case implies
$$
charge(T) = charge([{\bf{p}}_n]) + charge([{\bf{p}}_{n-1}])+\cdots+charge([{\bf{p}}_2])\,.
$$

Lapointe and Pinto naturally define a complementary $k$-cocharge statistic as well.  Although their statistics are not obviously non-negative, in the case of standard $k$-tableaux, the non-negativity follows from compatibility with the weak bijection~\cite{MR3115329}.  The non-negativity of $k$-(co)charge for any semi-standard $k$-tableaux was unresolved.

In private communication with Morse \cite{DMComm} (see also~\cite{kschurbook}), she defined manifestly non-negative statistics on $k$-tableaux and conjectured them to be the $k$-(co)charge.  Here we prove her conjecture and as a consequence prove that the $k$-(co)charge of any semi-standard $\ktableau$ is non-negative.  

\section{Related Work}

Since the inception of $k$-atoms in~\cite{[LLM]}, many articles concerning $s_\lambda^{(k)}(x;t)$ have appeared.  While most consider only the parameterless case when $t=1$, in addition to the work of Lapointe-Pinto just discussed, there have been a number of other achievements in full generality.  In~\cite{[B1]}, Blasiak conjectures that $k$-atoms can be characterized by catabolizability conditions and connects them to representation theory.  In \cite{[MSaffinecharge]}, $K_{\lambda\mu}^{(k)}(0,t)$ is shown to be an expression over solvable lattice models by Nakayashiki and Yamada.  The intense study~\cite{AB} of covers in the Bruhat order on the affine symmetric group sheds light on the conjectured characterization for $k$-atoms~\cite{[LLMS]} as generating functions for marked saturated chains.  The work of \cite{DM3} introduces a notion of affine charge on affine Bruhat counter-tableaux, objects in bijection with $\ktableaux$, and definitively proves that Macdonald polynomials and quantum and affine Schubert calculus are interconnected. 

\section{Prelimiaries}





A {\it{composition}} $\alpha = (\alpha_1, \alpha_2, \ldots, \alpha_m)$ is a vector of positive integers.  A {\it{partition}} $\lambda = (\lambda_1, \lambda_2, \ldots,\lambda_m)$ is a composition such that $\lambda_1 \geq \lambda_2 \geq \cdots \geq \lambda_m$.  The length of $\lambda$, denoted $\ell(\lambda)$, is the number of parts in $\lambda$, and the sum of the parts of $\lambda$ is denoted $|\lambda|$.  A partition of length $m$ whose parts are all $1$ will be denoted $(1^m)$.  A function on a partition $\lambda$ that we will use is
$$
n(\lambda) = \sum_{i=1}^{\ell(\lambda)} (i-1)\lambda_i\,.
$$

Every partition $\lambda$ has a corresponding {\it{Ferrer's diagram}}, which has $\lambda_i$ lattice cells in the $i^{th}$ row for $1 \leq i \leq \ell(\lambda)$.

\begin{example}
\label{partitionexample}
The partition $\lambda = (5,3,2,2,1,1)$ has the corresponding Ferrer's diagram
$$
{\text{\tiny{\tableau[sbY]{ \cr \cr & \cr & \cr & & \cr & & & & }}}}\,.
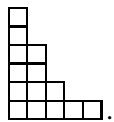
$$
\end{example}

If $\lambda \subseteq \mu$, then the skew shape $\mu/\lambda$ are those cells of $\mu$ which are not in $\lambda$.  Any cell $c$ of a Ferrer's diagram located in the $i^{th}$ row from the bottom and $j^{th}$ column from the left can be written as $c = (i,j)$.  For a given cell $c = (i,j)$ of a partition $\lambda$, the {\it{hook-length}} of $c$, $h_{\lambda}(c)$, is the number of cells in the $i^{th}$ row to the right of $c$ plus the number of cells in the $j^{th}$ column above $c$ plus 1 to include $c$.  

It is certain subsets of the set of partitions that are central in our study.  For a positive integer $n > 1$, an {\it{$n$-core}} is a partition whose shape has no cells of hook-length $n$.  Given any cell $c = (i,j)$ of an $n$-core, the {\it{$n$-residue}} of $c$, or the $res(c)$, is $(j-i) \mod n$.  This tells us that the $n$-core $\lambda$ has cells whose residues are periodically labelled with $0, 1, \ldots, n-1$, where zeros are the residues of the cells on the main diagonal.

\begin{example}
\label{coreexample}
An example of a $5$-core whose cells are labelled with $5$-residues is
$$
{\text{\tiny{\tableau[sbY]{ 1 \cr 2 \cr 3 & 4 \cr 4 & 0 & 1 \cr 0 & 1 & 2 & 3 & 4 & 0 & 1}}}}\,.
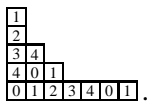
$$
\end{example}

A crucial proposition on cores that we will use comes from the work of L. Lapointe, A. Lascoux and J. Morse \cite{[LLM], [LMcore]}.  We say that a cell $(i,j)$ of a partition $\lambda$ is called an {\it{extremal cell}} if $(i+1,j+1) \not\in \lambda$.

\begin{proposition}
\label{lmcoreprop}
\cite{[LMcore]}
Let $\lambda$ be a $n$-core, where $c$ and $c'$ are extremal cells of $\lambda$ with the same $n$-residue.  
\begin{enumerate}
\item If $c'$ is weakly north-west of $c$ and $c$ is at the end of its row, then $c'$ is at the end of its row.  
\item If $c'$ is weakly south-east of $c$ and $c$ is at the top of its column, then $c'$ is at the top of its column. 
\end{enumerate}
\end{proposition}

One immediate consequence of Proposition \ref{lmcoreprop} is a simple remark that will be quite useful.  For any $n$-core $\lambda$, an {\it{addable corner}} of $\lambda$ is a cell $(i,j) \not\in \lambda$ with $(i,j-1),(i-1,j) \in \lambda$, and a {\it{removable corner}} of $\lambda$ is a cell $(i,j) \in \lambda$ with $(i,j+1),(i+1,j) \not\in \lambda$.

\begin{remark}
\label{lmcoreremark}
\cite{[LMcore]}
An $n$-core $\lambda$ never has both a removable corner and an addable corner of the same $n$-residue.
\end{remark}

For a positive integer $k$, a special filling of a $(k+1)$-core, called {\it$\ktableaux$}, was introduced in \cite{LMcore} to describe Pieri-type rules which the {\it $k$-Schur functions} satisfy at $t=1$.  Furthermore, at $t=1$, the {\it dual $k$-Schur functions} are the generating functions for these $\ktableaux$ of a given $(k+1)$-core shape.

\begin{definition}
\label{ktabdefinition}
For a positive integer $k$, let $\lambda$ be a $(k+1)$-core with $m$ $k$-bounded hooks and let $\alpha = (\alpha_1, \alpha_2, \ldots, \alpha_r)$ be a composition of $m$, where no $\alpha_i$ is larger than $k$.  A $k$-tableau of shape $\lambda$ and weight $\alpha$ is a tableau of shape $\lambda$ filled with integers $1,2,\ldots,r$ such that the collection of cells filled with letter $i$ are labeled by exactly $\alpha_i$ distinct $(k+1)$-residues.
\end{definition}

\begin{example}
\label{ktabexample0}
For $k=3$, a $\ktableau$ $T$ of shape $4$-core $(5,2,1)$ and weight $(2,2,2)$ is
$$
T = {\footnotesize{\tableau[scY]{ 3_2 \cr 2_3 & 3_0 \cr 1_0 & 1_1 & 2_2 & 2_3 & 3_0}}}
$$
We have labeled the $(k+1)$-residue's of the cells as the sub-scripts on the letters filling them.
\end{example}


\begin{example}
\label{ktabexample}
For $k=3$, the only two $\ktableau$ of weight $(3,2,1)$ are
$$
{\footnotesize{\tableau[scY]{2_3 & 2_0 & 3_1 \cr 1_0 & 1_1 & 1_2 & 2_3 & 2_0 & 3_1}}} \hspace{0.2in} \& \hspace{0.2in} {\footnotesize{\tableau[scY]{3_2 \cr 2_3 & 2_0 \cr 1_0 & 1_1 & 1_2 & 2_3 & 2_0}}}
$$
\end{example}
A key focus in our work will be on standard $\ktableaux$, those of weight $(1^m)$.

\begin{example}
\label{stdktabexample}
For $k=2$, all standard $\ktableau$ of weight $(1^4)$ are
$$
{\footnotesize{\tableau[sbY]{3_2 & 4_0 \cr 1_0 & 2_1 & 3_2 & 4_0}}} \hspace{0.2in} {\footnotesize{\tableau[sbY]{4_1 \cr 3_2 \cr 1_0 & 2_1 & 3_2}}} \hspace{0.2in} {\footnotesize{\tableau[sbY]{3_1 \cr 2_2 \cr 1_0 & 3_1 & 4_0}}} \hspace{0.2in} {\footnotesize{\tableau[sbY]{4_0 \cr 3_1 \cr 2_2 & 4_0 \cr 1_0 & 3_1}}}\,.
$$
\end{example}

One intrinsic property of a $\ktableau$ is in its shape when the parameter $k$ is taken sufficiently large.  A $\ktableau$ of shape $(k+1)$-core $\lambda$ has no cells of hook-length larger than $k$ if and only if $k > \lambda_1 + \ell(\lambda) - 2$.  For these large values of $k$, the $\ktableau$ is a semi-standard tableau.  As a result, when a $\ktableau$ has shape $(k+1)$-core $\lambda$ whose cells have hook-length less than $k+1$, the $\charge$ and a $\cocharge$ from \cite{LSfoulkes} can be computed on that $\ktableau$.  Our focus is on two statistics which apply to a $\ktableau$ for any $k>0$.


\section{$\kcocharge$ of a standard $\ktableau$}
\label{stdkcchsection}


From here on, we will always assume that $k$ is a positive integer, and we will label the $(k+1)$-residues, or just residues, of the cells of a $\ktableau$ as the sub-scripts on the letters filling them.  Furthermore, the weight of our $\ktableau$ will be a partition $\alpha$ whose parts are not larger than $k$.  The $\kcocharge$ statistic on $\ktableaux$ is first described for a standard $\ktableau$.  Important to the definition is the number of diagonals of a specific residue between two cells.

\begin{definition}
\label{diagdef}
Given two cells $c_1$ and $c_2$ of a $(k+1)$-core, let $diag(c_1,c_2)$ be the number of diagonals of reside $r$ that are strictly between $c_1$ and $c_2$ where $r$ is the residue of the lower cell.
\end{definition}

\begin{example}
\label{diagex}
For $k=4$,  a standard $k$-tableau of weight $(1^{9})$ is
$$
T = {\text{\footnotesize{\tableau[scY]{8_2 \cr 5_3 & 7_4 \cr 4_4 & 6_0 \cr 1_0 & 2_1 & 3_2 & 5_3 & 7_4 & 9_0}}}} \hspace{0.1in} \Longrightarrow \hspace{0.1in} diag(4_4,3_2) = 0, \emph{ and  } diag(8_2,(1,5)) = 1.
$$
\end{example}

When it is well-defined to do so, functions defined with a cell as input can instead take a letter as input.  In particular, for standard $\ktableaux$ it is natural to discuss the residue of a specific letter (since any cell containing that letter has the same residue) instead of the residue of a specific cell.  

\begin{definition}
\label{LPstdcch}
Given a standard $\ktableau$ $T$ of weight $(1^m)$, the lowest occurrence of $i$ will be denoted $\idown$, for $1 \leq i \leq m$.  Define the index vector $L(T) = [L_1, L_2, \ldots, L_m]$ recursively by setting $L_1 = 0$, and
$$
L_i = 
\begin{cases}
L_{i-1} + 1 + diag(\idown, (i-1)^{\downarrow}) & \text{if } (i-1)^{\downarrow} \text{ is strictly below } \idown \\
L_{i-1} - diag(\idown, (i-1)^{\downarrow}) & \text{otherwise}
\end{cases}
$$
for $2 \leq i \leq m$.  The $\kcocharge$ of $T$ is the sum of the entries of $L(T)$,
$$
\kcocharge(T) = \sum_{i=1}^m L_i\,.
$$
\end{definition}

\begin{example}
\label{LPcchex}
For the $\ktableau$ $T$ of Example \ref{diagex}, Table \ref{tab:stdcochargeextable} shows us that the $\kcocharge(T) = 13$.
\end{example}



In contrast to the cocharge from \cite{LSfoulkes}, Definition~\ref{LPstdcch} does not suggest that the $\kcocharge$ of a standard $\ktableau$ must be non-negative.  However, the compatability of $k$-cocharge with the $k$-shape poset for standard $k$-tableau $T$ implies that $\kcocharge(T)\geq 0$ in this case~\cite{MR3115329}.  Lapointe-Pinto also defined the statistic for semi-standard $k$-tableaux, but serious obstructions to extending the compatibility between $\kcocharge$ and the weak bijection in the general case left the non-negativity of $k$-cocharge unresolved.

In private communication, Morse provided a different, manifestly non-negative, statistic on $k$-tableaux and she conjectured it to be equivalent to Definition~\ref{LPstdcch}.  We recall her definition for standard $k$-tableaux and start by proving her conjecture in this case.  This requires a {\it residue order} on a $\ktableau$.

\begin{definition}
\label{ktabresorderdef}
Let $T$ be a standard $\ktableau$.  The low $T$-residue order of $\{0,1,\ldots,k\}$ is defined by
$$
x > x+1 > \cdots > k > 0 > 1 > \cdots > x-1\,,
$$
where $x$ is the residue of the lowest addable cell of $T$.  If the standard $\ktableau$ $T$ is of weight $(1^m)$, then $T_{\leq i}$ are those cells of $T$ filled with a letter $j \leq i$, for $1 \leq i \leq m$,
\end{definition}

\begin{example}
\label{ktabresorderex}
For the standard $\ktableau$ $T$ of Example \ref{diagex}, we see that
$$
T_{\leq 5} = {\text{\footnotesize{{\tableau[scY]{5_3 \cr 4_4 \cr 1_0 & 2_1 & 3_2 & 5_3}}}}}
$$
The low $T$-residue order is $2 > 3 > 4 > 0 > 1$, and the low $T_{\leq 5}$-residue order is $4 > 0 > 1 > 2 > 3$, which is also the low $T_{\leq 6}$-residue order.
\end{example}


Morse's statistic~\cite{DMComm, kschurbook} for standard $\ktableau$ also involves an index vector, defined recursively using the residue order of Definition \ref{ktabresorderdef}.

\begin{definition}
\label{cochargeindexvector}
Given a standard $\ktableau$ $T$ of weight $(1^m)$, define the index vector $M(T) = [M_1,\ldots,M_m]$ recursively by setting $M_1 = 0$, and
$$
M_i = 
\begin{cases}
M_{i-1} + 1 & \text{if }res(i) > res(i-1) \\
M_{i-1} & \text{otherwise}
\end{cases}
$$
where $res(i)$ and $res(i-1)$ is compared using the low $T_{\leq i}$-residue order, for $2 \leq i \leq m$.
\end{definition}


\begin{theorem}
\label{cochargeLM}
For a standard $k$-tableau $T$ of weight $(1^m)$,
$$
\kcocharge(T) = \sum_{i = 1}^m \left( M_i + \text{diag}(\idown,c_{(i)}) \right)\,,
$$
where $c_{(i)}$ is the lowest addable cell of $T_{\leq i}$.
\end{theorem}

\begin{proof}
The proof is by induction on the weight of $T$.  The base case of $i=1$ is trivial as we get $L_1 = M_1$.  Next suppose that $L_{i-1} = M_{i-1} + diag({(i-1)}^{\downarrow}, c_{i-1})$ and the last cell in the bottom row of $T_{\leq i}$ is of residue $r$.  The inductive step is proved by considering cases.  Suppose there is a lowest cell $c$ of residue $r$ that is not in the bottom row of $T_{\leq i}$, and $\imdown$ is weakly north-west of $c$ while $\idown$ is weakly south-east of $c$, and $\idown$ is not in the bottom row of $T_{\leq i}$.

On the one hand if $res(i) > res(i-1)$, then $M_i = M_{i-1} + 1 + diag(\idown, c_{(i)})$.  If $\idown$ is not in the bottom row of $T_{\leq i}$, then $c_{(i-1)} = c_{(i)}$.  This says
\begin{align*}
L_i &= L_{i-1} - diag(\idown, \imdown) \\
      &= M_{i-1} + diag(\imdown, c_{(i-1)}) - diag(\idown, \imdown) \\
      &= M_{i-1} + diag(\imdown, c_{(i)}) - diag(\idown, \imdown).
\end{align*}
Now if $c$ is the lowest extremal cell of residue $r$ that is not in the bottom row, then $diag(\imdown, c_{(i)}) = diag(\imdown, c) + 1 + diag(c, c_{(i)})$.  Since the $res(i) > res(i-1)$ and $\idown$ is weakly south-east of $c$, then $diag(c, c_{(i)}) = diag(\idown, c_{(i)})$.  Thus, $diag(\imdown, c_{(i)}) = diag(\imdown,c) + 1 + diag(\idown,c_{(i)})$.  This gives
\begin{align*}
L_i &= M_{i-1} + diag(\imdown, c_{(i)}) - diag(\idown, \imdown) \\
      &= M_{i-1} + 1 + diag(\idown,c_{(i)}) + diag(\imdown, c) - diag(\idown, \imdown).
\end{align*}
If $\imdown$ is weakly north-west of $c$ and $\idown$ is weakly south-east of $c$, then $diag(\imdown,c) = diag(\imdown, \idown)$.  This along with the fact that $c_{(i-1)} = c_{(i)}$, gives $L_i = M_i$.

On the other hand if the $res(i) < res(i-1)$, then $M_i = M_{i-1} + diag(\idown, c_{(i)})$.  The proof for this case follows almost similar to the proof for the case above, except this time $diag(\imdown, c) + 1 = diag(\idown,\imdown)$ because $res(i) < res(i-1)$.

The other cases, which are left to the reader to prove, are
\begin{enumerate}
\item $\idown$ is weakly south-east of ${(i-1)}^{\downarrow}$ in $T_{\leq i}$.  
\begin{enumerate}

\item ${(i-1)}^{\downarrow}$ is weakly south-east of any extremal cell of residue $r$ that is not in the bottom row of $\Tleqi$, and $\idown$ is in the bottom row of $\Tleqi$.

\item There is a lowest cell $c$ of residue $r$ that is not in the bottom row of $T_{\leq i}$, and both $\imdown$ and $\idown$ are weakly north-west of $c$.
\end{enumerate}
\item $\imdown$ is south-east of $\idown$, along with the sub-cases as above.
\end{enumerate}
\end{proof}

\begin{example}
\label{stdcochargeex}
For the $\ktableau$ $T$ of weight $(1^{9})$ from Example \ref{diagex}, Table \ref{tab:stdcochargeextable} shows us that 
$$
\sum_{i=1}^{9} \left( M_i + diag(\idown, c_{(i)}) \right) = 13\,.
$$
This agrees with Theorem \ref{cochargeLM} since Example \ref{LPcchex} told us that the $\kcocharge(T) = 13$.
\end{example}

\begin{table}[h]
\renewcommand{\arraystretch}{1.4}
\begin{center}
\footnotesize{
\begin{tabular}{| c | c | c | c | c | c |}
\hline 
$i$ & diag$(\idown, (i-1)^{\downarrow})$ & $L_i$ & Low $T_{\leq i}$-residue order & $M_i$ & diag$(i^{\downarrow},c_{(i)})$ \\ 
\hline
1 & - & 0 & - & 0 & 0 \\
\hline 
$2$ & $0$ & $0 - 0 = 0$ & $2 > 3 > 4 > 0 > 1$ & $0$ & $0$ \\
\hline
$3$ & $0$ & $0 - 0 = 0$ & $3 > 4 > 0 > 1 > 2$ & $0$ & $0$ \\
\hline
$4$ & $0$ & $0 + 1 + 0 = 1$ & $3 > 4 > 0 > 1 > 2$ & $1$ & $0$ \\
\hline
$5$ & $0$ & $1 - 0 = 1$ & $4 > 0 > 1 > 2 > 3$ & $1$ & $0$ \\
\hline
$6$ & $0$ & $1 + 1 + 0 = 2$ & $4 > 0 > 1 > 2 > 3$ & $2$ & $0$ \\
\hline
$7$ & $0$ & $2 - 0 = 2$ & $0 > 1 > 2 > 3 > 4$ & $2$ & $0$ \\
\hline
$8$ & $1$ & $2 + 1 + 1 = 4$ & $0 > 1 > 2 > 3 > 4$ & $3$ & $1$ \\
\hline
$9$ & $1$ & $4 - 1 = 3$ & $1 > 2 > 3 > 4 > 0$ & $3$ & $0$ \\
\hline
\end{tabular}
}
\end{center}
\caption{$\kcocharge$ of $T$ from Example \ref{diagex}} 
\label{tab:stdcochargeextable}
\end{table}

\section{$\kcocharge$ of a semi-standard $\ktableau$}
The Lascoux and Sch\"utzenberger cocharge statistic on a semi-standard tableau is an extension of the cocharge statistic on a standard tableau.  To do this for $k$-cocharge on a $k$-tableau $T$, a method for making an appropriate choice of standard sequences on $T$ is required.


\begin{definition}
\label{stdseqchoicedef}
Let $T$ be a semi-standard $\ktableau$.  Construct each standard sequence iteratively by 
\begin{enumerate}
\item first choosing the right-most cell in $T$ filled with a $1$ which has not been chosen in a previous standard sequence.

\item Having chosen an $i-1$, of some residue $r$, the appropriate choice of $i$ will be determined by considering the residues of only those $i$'s which haven't been chosen in a previous standard sequence.  Label all the residues $0,1,\ldots,k$ on a circle clockwise.  The appropriate choice of $i$ is the one whose residue is closest to $r$ reading counter-clockwise from $r$ on the circle.
\end{enumerate}
\end{definition}

Observe that Definition \ref{stdseqchoicedef} is well defined since the collection of cells filled with the letter $i$ are labeled by exactly $\alpha_i$ distinct residues, where $\alpha$, the weight of the $\ktableau$, is a partition whose parts are not larger than $k$.  Definition \ref{stdseqchoicedef} also shows us a way to iteratively compute a set of standard sequences for a given semi-standard $\ktableau$.

\begin{example}
\label{stdseqchoiceex}
For $k=4$, a $\ktableau$ of weight $(2,2,2,2,2,2,1)$ is
$$
T = {\text{\footnotesize{\tableau[scY]{\tf 7_0 \cr \tf 6_1 \cr \tf 5_2 & 6_3 \cr \tf 3_3 & 4_4 & \tf 7_0 \cr \tf 2_4 & 3_0 & 5_1 & \tf 5_2 & 6_3 \cr 1_0 & \tf 1_1 & 2_2 & \tf 3_3 & 4_4 & \tf 4_0 & 5_1 & \tf 5_2 & 6_3}}}} 
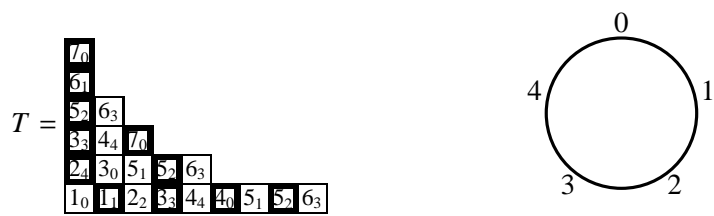
\hspace{1in}
\begin{tikzpicture}[baseline = -0.1in]
\draw[very thick] (0,0) circle (1);
\node at (0,1.2) {$0$};
\node at (-1.151, 0.309) {$4$};
\node at (-0.709, -0.95) {$3$};
\node at (0.709,-0.95) {$2$};
\node at (1.151,0.309) {$1$};
\end{tikzpicture}
$$
The bold cells of $T$ shows the first standard sequence of cells using Definition \ref{stdseqchoicedef}.  The set of residues labeling $5$ in $T$ is $\{1,2\}$, and the residue of $4$ in the first standard sequence is $0$.  Therefore, the choice of residue $2$ from $\{1,2\}$ is made because $2$ is closer to $0$ than $1$ when reading counter-clockwise on the above circle labeled with all the residues.  The cells which are not in bold form the second standard sequence of cells.
\end{example}


The $\kcocharge$ definition for a standard $\ktableau$ extends to a semi-standard $\ktableau$ $T$ by summing over the standard sequences of $T$.  

\begin{definition}
\label{LPsstdcch}
Let $T$ be a semi-standard $\ktableau$.  For each standard sequence $s$ of $T$, define the index vector $L^{(s)}(T) = [L_1^{(s)},\ldots, L_{\ell(s)}^{(s)}]$ recursively by setting $L_1^{(s)} = 0$, and
$$
L_{i}^{(s)} = 
\begin{cases}
L_{i-1}^{(s)} + 1 + diag(\idown, (i-1)^{\downarrow}) & \text{if } (i-1)^{\downarrow} \text{ is strictly below } \idown \text{ in } s \\
L_{i-1}^{(s)} - diag(\idown, (i-1)^{\downarrow}) & \text{otherwise}
\end{cases}
$$
for $2 \leq i \leq \ell(s)$, where $\ell(s)$ is the length of $s$.  If $S$ is the set of all standard sequences, then the $\kcocharge$ of $T$ is 
$$
\kcocharge(T) = \sum_{s \in S} \, \sum_{i=1}^{\ell(s)} L_i \,.
$$
\end{definition}

\begin{example}
\label{LPsscchex}
For the semi-standard $\ktableau$ $T$ of Example \ref{stdseqchoiceex}, Definition \ref{LPsstdcch} says that the $\kcocharge(T) = 15$.
\end{example}



Definition~\ref{LPsstdcch} does not make the non-negativity of $k$-cocharge apparent.  Again, we instead match the definition with Morse's semi-standard statistic 
defined as follows.

\begin{definition}
\label{ssktabresorder}
For a semi-standard $\ktableau$ $T$ which has cells of residue $r$ and filled with $i$, define $T_{\leq i_r}$ as those cells of $T$ filled with $j \leq i$, where $j$ is in the same standard sequence as $i$.  The low $T_{\leq i_r}$-residue order of $\{0,1,\ldots,k\}$ is defined by 
$$
x > x+1 > \cdots > k > 0 > 1 > \cdots > x-1
$$
where $x$ is the residue of the lowest addable cell of $T_{\leq i_r}$.
\end{definition}

\begin{example}
\label{semistandardresidueorder}
For $k =  4$, since the second standard sequence of the semi-standard $k$-tableau $T$ of Example \ref{stdseqchoiceex} are the bold cells of
$$
{\text{\footnotesize{\tableau[scY]{7_0 \cr 6_1 \cr 5_2 & \tf 6_3 \cr 3_3 & \tf 4_4 & 7_0 \cr 2_4 & \tf 3_0 & \tf 5_1 & 5_2 & \tf 6_3 \cr \tf 1_0 & 1_1 & \tf 2_2 & 3_3 & \tf 4_4 & 4_0 & \tf 5_1 & 5_2 & \tf 6_3}}}} \hspace{0.3in} \Longrightarrow \hspace{0.3in} T_{\leq 5_1} = {\text{\footnotesize{\tableau[scY]{\bl \cr \bl & \bl \cr \bl & 4_4 & \bl \cr \bl & 3_0 & 5_1 & \bl & \bl \cr 1_0 & \bl & 2_2 & \bl & 4_4 & \bl & 5_1 & \bl & \bl}}}}
$$
The lowest addable cell of $T_{\leq 5_1}$ has residue $2$, so the low $T_{\leq 5_1}$-residue order is $2 > 3 > 4 > 0 > 1$.  Observe that the $5$ of residue $2$ is in the first standard sequence of $T$.  Since the lowest addable cell of $T_{\leq 5_2}$ is of residue $3$, then the $T_{\leq 5_2}$-residue order is $3 > 4 > 0 > 1 > 2$.
\end{example}

The alternate form of the $\kcocharge$ for a semi-standard $\ktableau$ involves multiple index vector's.  Each index vector corresponds to a standard sequence of the semi-standard $\ktableau$, and it is defined recursively using the residue order of Definition \ref{ssktabresorder}.

\begin{definition}
\label{Jsstdcch}
Let $T$ be a semi-standard $k$-tableau.  For each standard sequence $s$ of $T$, define the index vector $M^{(s)}(T) = [M_1^{(s)}, \ldots, M_{\ell(s)}^{(s)}]$ recursively by setting $M_1^{(s)} = 0$, and
$$
M_{i}^{(s)} = 
\begin{cases}
M_{i-1}^{(s)} + 1 & \text{if } res(i) > res(i-1) \\
M_{i-1}^{(s)} & otherwise
\end{cases}
$$
where $res(i)$ and $res(i-1)$ are compared using the low $T_{\leq i_r}$-residue order for each $i \in s$.
\end{definition}







For a given semi-standard $\ktableau$ $T$, summing over the standard sequences of $T$ gives us a corollary to Theorem \ref{cochargeLM}.

\begin{corollary}
\label{sscochargeLM}
Let $T$ be a semi-standard $\ktableau$.  If $S$ is the set of all standard sequences of $T$, then the non-negative
$$
\kcocharge(T) = \sum_{s \in S} \sum_{i = 1}^{\ell(s)} \left(M_i^{(s)} + diag(\idown, c_{(i)}) \right)\,,
$$
where $\idown$ is the lowest occurrence of $i$ in $s$ and $c_{(i)}$ is the lowest addable cell of $T_{\leq i_r}$.
\end{corollary}

\begin{example}
\label{resorderofstdseqex}
For the $\ktableau$ $T$ of Example \ref{stdseqchoiceex}, we have the two standard sequences by Definition \ref{stdseqchoicedef}
$$
{\text{\footnotesize{\tableau[scY]{\tf 7_0 \cr \tf 6_1 \cr \tf 5_2 & 6_3 \cr \tf 3_3 & 4_4 & \tf 7_0 \cr \tf 2_4 & 3_0 & 5_1 & \tf 5_2 & 6_3 \cr 1_0 & \tf 1_1 & 2_2 & \tf 3_3 & 4_4 & \tf 4_0 & 5_1 & \tf 5_2 & 6_3}}}} 
\hspace{0.5in} \& \hspace{0.5in}
{\text{\footnotesize{\tableau[scY]{ \cr \cr & \tf 6_3 \cr & \tf 4_4 & \cr & \tf 3_0 & \tf 5_1 & & \tf 6_3 \cr \tf 1_0 & & \tf 2_2 & & \tf 4_4 & & \tf 5_1 & & \tf 6_3}}}} 
$$
Corollary \ref{sscochargeLM} tells us that the $\kcocharge(T) = 16$.
\end{example}


A $\ktableau$ of shape $(k+1)$-core $\lambda$ has no cells of hook-length larger than $k$ if and only if $k > \lambda_1 + \ell(\lambda) - 2$.  For these large values of $k$, the $\ktableau$ is a semi-standard tableau of the same weight.  An observation of Corollary \ref{sscochargeLM} is that when the $\ktableau$ has shape $(k+1)$-core $\lambda$ whose cells have hook-length less than $k+1$, the $\kcocharge$ of the $\ktableau$ is the $\cocharge$ from \cite{LSfoulkes}.

\section{$\kcharge$ of a $\ktableau$}

In this section we begin by recalling the definition of $\kcharge$ of a standard $\ktableau$ from \cite{MR3115329}.  Our goal is to show that this definition is equivalent to Morse's non-negative formulation~\cite{DMComm, kschurbook}.

\begin{definition}
\cite{MR3115329}
\label{stdLch}
Given a standard $\ktableau$ $T$ of weight $(1^m)$, the highest occurrence of $i$ will be denoted $\iup$, for $1 \leq i \leq m$.  Define the index vector $I(T) = [I_1,I_2, \ldots, I_m]$ recursively by setting $I_1 = 0$, and
$$
I_i = 
\begin{cases}
I_{i-1} + 1 + diag(\iup, {(i-1)}^{\uparrow}) & \text{if } \iup \text{ is east of } {(i-1)}^{\uparrow} \\
I_{i-1} - diag(\iup, {(i-1)}^{\uparrow}) & \text{otherwise.}
\end{cases}
$$
for $2 \leq i \leq m$.  The $\kcharge$ of $T$ is the sum of the entries of $I(T)$,
$$
\kcharge(T) = \sum_{i=1}^m I_i.
$$
\end{definition}

\begin{example}
\label{LPchex}
For $k = 4$, recall the $\ktableau$ from Example \ref{diagex} is
$$
T = {\text{\footnotesize{\tableau[scY]{ 8_2 \cr 5_3 & 7_4 \cr 4_4 & 6_0 \cr 1_0 & 2_1 & 3_2 & 5_3 & 7_4 & 9_0 }}}}
$$
Applying Definition \ref{stdLch} for this $T$, Table \ref{tab:Jstdchargeextable} shows us that the $\kcharge(T) = 21$.
\end{example}



Similar to the $\kcocharge$ of Definition \ref{LPstdcch}, it is not immediately clear that the $\kcharge$ is a non-negative integer and we use Morse's non-negative formulation~\cite{DMComm, kschurbook} which arises by altering the residue order in the natural way.

\begin{definition}
\label{highresorderdef}
The high $T$-residue order of $\{0,\ldots,k\}$ is defined by
$$
x > x-1 > \ldots > 0 > k > \ldots > x+1\,,
$$
where $x$ is the residue of the highest addable corner of $T$.  
\end{definition}

\begin{example}
For the standard $\ktableau$ $T$ of Example \ref{LPchex}, we see that
$$
T_{\leq 5} = {\text{\footnotesize{\tableau[scY]{5_3 \cr 4_4 \cr 1_0 & 2_1 & 3_2 & 5_3}}}}
$$
The high $T$-residue order is $0 > 1 > 2 > 3 > 4$, and the high $T_{\leq 5}$-residue order is $2 > 1 > 0 > 4 > 3$, which is also the high $T_{\leq 6}$-residue order and the high $T_{\leq 7}$-residue order.
\end{example}


\begin{definition}
\label{chargeindexvector}
Given a standard $\ktableau$ $T$ of weight $(1^m)$, define the index vector $J(T) = [J_1, \ldots, J_m]$, starting from $J_1 = 0$, by setting for $i = 2, \ldots, m$, 
$$
J_i = 
\begin{cases}
J_{i-1} + 1 & \text{if } res(i) > res(i-1) \\
J_{i-1} & \text{otherwise},
\end{cases}
$$
where $res(i)$ and $res(i-1)$ is compared using the high $T_{\leq i}$-residue order.
\end{definition}


\begin{example}
\label{chstdJex}
For the $\ktableau$ $T$ of weight $(1^{9})$ of Example \ref{LPchex}, we see that Table \ref{tab:Jstdchargeextable} gives us that
$$
\sum_{i=1}^{9} \left(J_i + diag(\iup, c^{(i)}) \right) = 21\,.
$$
Example \ref{LPchex} also told us that this sum is $\kcharge(T)$.
\end{example}

\begin{table}[t]
\renewcommand{\arraystretch}{1.4}
\begin{center}
\footnotesize{
\begin{tabular}{| c | c |c | c | c | c |}
\hline
$i$ & diag$(\iup, (i-1)^{\uparrow})$ & $I_i$ & High $T_{\leq i}$-residue order & $J_i$ & diag$(\iup, c^{(i)})$\\ 
\hline
$1$ & - & 0 & - & $0$ & 0 \\
\hline 
$2$ & 0 & $0 + 1 + 0 = 1$ & $4 > 3 > 2 > 1 > 0$ & $1$ & $0$\\
\hline
$3$ & 0 & $1 + 1 + 0 = 2$ & $4 > 3 > 2 > 1 > 0$ & $2$ & $0$\\
\hline
$4$ & 0 & $2 - 0 = 2$ & $3 > 2 > 1 > 0 > 4$ & $2$ & $0$ \\
\hline
$5$ & 0 & $2 - 0 = 2$ & $2 > 1 > 0 > 4 > 3$ & $2$ & $0$\\
\hline
$6$ & 0 & $2 + 1 - 0 = 3$ & $2 > 1 > 0 > 4 > 3$ & $3$ & $0$\\
\hline
$7$ & 0 & $3 - 0 = 3$ & $2 > 1 > 0 > 4 > 3$ & $3$ & $0$\\
\hline
$8$ & 0 & $3 - 0 = 3$ & $1 > 0 > 4 > 3 > 2$ & $3$ & $0$\\
\hline
$9$ & 1 & $3 + 1 + 1 = 5$ & $1 > 0 > 4 > 3 > 2$ & $4$ & $1$\\
\hline
\end{tabular}
}
\end{center}
\caption{$\kcharge$ of $T$ from Example \ref{diagex}} 
\label{tab:Jstdchargeextable}
\end{table}

Just as we showed a non-negative reformulation the $\kcocharge$ of a standard $\ktableau$ in Section \ref{stdkcchsection}, we will show that Definition \ref{chargeindexvector} is a non-negative reformulation of the $\kcharge$ of a standard $\ktableau$.  In fact, we will show that it gives a non-negative reformulation of the $\kcharge$ for a semi-standard $\ktableau$.  We begin with some key Lemma's.

\begin{lemma}
\label{betweendiags}
If $T$ is a standard $\ktableau$ of weight $(1^m)$, then any diagonal of the same residue as $i$ and between the diagonals with $\iup$ and $\idown$ must also have an $i$, for $1 \leq i \leq m$.
\end{lemma}

\begin{proof}
For each $1 \leq i \leq m$, let $c$ be the highest cell on any diagonal of $\Tleqi$, where $c$ is of the same reside as $\iup$ and the diagonal with $c$ is between the diagonals with $\iup$ and $\idown$.  If $c$ is of the same residue as $\iup$ and $\idown$, then Proposition \ref{lmcoreprop} tells us that $c$ is at the end of its row and at the top of its column in $\Tleqi$.  Remark \ref{lmcoreremark} tells us that $c$ must contain an $i$.
\end{proof}

\begin{lemma}
\label{diaglemma}
Let $T$ be a standard $\ktableau$ of weight $(1^m)$.  If $\beta_i$ is the number of cells of $T$ filled with $i$, then 
$$
\beta_i + diag(\iup, c^{(i)}) + diag(\idown, c_{(i)})
$$
is the number of diagonals of residue $res(i)$ in $\Tleqi$, for $1 \leq i \leq m$.
\end{lemma}

\begin{proof}
In $\Tleqi$, there is a diagonal of residue $res(c^{(i)})$ above the diagonal with $\iup$ if and only if there is a diagonal of residue $res(i)$ above the diagonal with $\iup$.  Similarly, in $\Tleqi$, there is a diagonal of residue $res(c_{(i)})$ below the diagonal with $\idown$ if and only if there is a diagonal of residue $res(i)$ below the diagonal with $\idown$.  Lemma \ref{betweendiags} tells us that $\beta_i$ is the number of diagonals of residue $res(i)$ between the diagonals with $\iup$ and $\idown$ in $\Tleqi$.
\end{proof}

Using Lemma \ref{diaglemma} and the following definition, we show the connection between Morse's statistics ~\cite{DMComm, kschurbook} of Definition's \ref{cochargeindexvector} and \ref{chargeindexvector}.

\begin{definition}
\label{intkcore}
The $k$-interior of a partition $\lambda$ is the sub-partition made of the cells of $\lambda$ with hook-length larger than $k$:
$$
\text{Int}^k(\lambda) = \{c \in \lambda \mid h_{\lambda}(c) > k\}\,.
$$
\end{definition} 
For a standard $\ktableau$ $T$ of shape $\lambda$ and weight $(1^m)$, observe that
$$
|Int^k(\lambda)| = \sum_{i=1}^m (\beta_i - 1)\,,
$$
where $\beta_i$ is the number of cells of $T$ filled with $i$.
\begin{theorem}
\label{JMequation}
Given a standard $\ktableau$ $T$ of weight $(1^m)$ and shape $\lambda$,
$$
\sum_{i=1}^m \left(J_i + diag(\iup, c^{(i)}) \right) = \frac{m(m-1)}{2} - |Int^k(\lambda)| - \sum_{i=1}^m \left(M_i + diag(\idown, c_{(i)}) \right)\,.
$$
\end{theorem}

\begin{proof}
The proof is by induction on the weight of $T$.  The case when $i=1$ is clear.  Using the induction hypothesis, we must show that
$$
J_{i+1} + diag((i+1)^{\uparrow}, c^{(i+1)}) = i - \beta_{i+1} + 1 - M_{i+1} - diag((i+1)^{\downarrow}, c_{(i+1)})\,,
$$
where $\beta_{i+1}$ is the number of cells of $T$ filled with $i+1$.

We first consider the case that $res(i+1) > res(i)$ under both the high and low $T_{\leq i+1}$-residue order.  Under this case we have $J_{i+1} = J_i + 1$, $M_{i+1} = M_i + 1$ and the number of diagonals of residue $res(i)$ is one more than the number of diagonals with residue $res(i+1)$ in $T_{\leq i+1}$.  This tells us that 
$
J_{i+1} + diag((i+1)^{\uparrow}, c^{(i+1)}) = J_i + 1 +  diag((i+1)^{\uparrow}, c^{(i+1)})\,.
$
The induction hypothesis tells us that $J_i + 1 +  diag((i+1)^{\uparrow}, c^{(i+1)})$ is equal to
\begin{equation}
\label{proofeq2}
i - 1 - \beta_{i} + 1 - M_i - diag(i^{\downarrow}, c_{(i)}) - diag(i^{\uparrow}, c^{(i)}) + 1 + diag((i+1)^{\uparrow}, c^{(i+1)})\,.
\end{equation}
Lemma \ref{diaglemma} tells us that $\beta_i + diag(i^{\downarrow}, c_{(i)}) + diag(i^{\uparrow}, c^{(i)})$ is the number of diagonals of residue $res(i)$ in $T_{\leq i}$.  The number of diagonals of residue $res(i)$ is one more than the number of diagonals of residue $res(i+1)$ in $T_{\leq i+1}$.  Lemma \ref{diaglemma} tells us again that the number of diagonals of residue $res(i+1)$ in $T_{\leq i+1}$ is $\beta_{i+1} + diag((i+1)^{\uparrow},c^{(i+1)}) + diag((i+1)^{\downarrow}, c_{(i+1)})$.
Furthermore, since $M_{i+1} = M_i + 1$, then \eqref{proofeq2} reduces to
$
i - \beta_{i+1} + 1 - M_{i+1} - diag((i+1)^{\downarrow}, c_{(i+1)})\,. 
$

The other cases, which are left to the reader to prove, are
\vspace{-0.1in}
\begin{enumerate}
\item $res(i+1) > res(i)$ under the high $T_{\leq i+1}$-residue order, and $res(i+1) < res(i)$ under the low $T_{\leq i+1}$-residue order. \vspace{-0.1in}
\item $res(i+1) < res(i)$ under the high $T_{\leq i+1}$-residue order, and $res(i+1) > res(i)$ under the low $T_{\leq i+1}$-residue order. \vspace{-0.1in}
\item $res(i+1) < res(i)$ under the high $T_{\leq i+1}$-residue order, and $res(i+1) < res(i)$ under the low $T_{\leq i+1}$-residue order. \vspace{-0.1in}
\end{enumerate}
\end{proof}

\begin{example}
For $k = 4$, recall that the standard $k$-tableau of weight $(1^{9})$ from Example \ref{diagex} is
$$
T = {\text{\footnotesize{\tableau[scY]{8_2 \cr 5_3 & 7_4 \cr 4_4 & 6_0 \cr 1_0 & 2_1 & 3_2 & 5_3 & 7_4 & 9_0}}}}
$$
For this $T$, Example's \ref{stdcochargeex} and \ref{chstdJex} tell us
$$
\sum_{i=1}^m \left(M_i + diag(\idown, c_{(i)}) \right) = 13 \hspace{0.2in} \& \hspace {0.2in} \sum_{i=1}^m \left(J_i + diag(\iup, c^{(i)}) \right) = 21\,.
$$
These equations along with the fact that $|Int^4((7,3,2,1,1))| = 2$ and $10(10-1)/{2} = 36$ agrees with Theorem \ref{JMequation}.
\end{example}

For a given semi-standard $\ktableau$ $T$, we can sum over the standard sequences of $T$ to generalize Theorem \ref{JMequation}.

\begin{theorem}
\label{JMequationssk}
Let $T$ be a semi-standard $\ktableau$ of weight $\mu$ and shape $\lambda$.  If $S$ is the set of all standard sequences of $T$, then
$$
\sum_{s \in S}\sum_{i=1}^{\ell(s)} \left(J_i^{(s)} + diag(\iup, c^{(i)}) \right) = n(\mu) - |Int^k(\lambda)| - \sum_{s \in S}\sum_{i=1}^{\ell(s)} \left(M_i^{(s)} + diag(\idown, c_{(i)}) \right),
$$
where $\iup$ and $\idown$ are the highest and lowest occurrences of $i$ in $s$, respectively, and $c^{(i)}$ and $c_{(i)}$ are the highest and lowest addable cells of $T_{\leq i_r}$, respectively.
\end{theorem}

\begin{proof}
Given the $\ktableau$ $T$, let $S$ be the set of all standard sequences from Definition \ref{stdseqchoicedef}.  Each standard sequence $s \in S$ contributes $\sum_{i=1}^{\ell(s)} (\beta_i -1)$ to $|Int^k(\lambda)|$, where $\beta_i$ is the number of cells of $T$ filled with $i \in s$.  Summing over all the standard sequences of $S$, we see that
$$
|Int^k(\lambda)| = \sum_{s \in S} \sum_{i=1}^{\ell(s)} (\beta_i - 1)\,,
$$
where $\lambda$ is the shape of $T$.

For the $n(\mu)$ term, each standard sequence $s \in S$ contributes $(\ell(s)(\ell(s) - 1))/2$ to $n(\mu)$.  Summing over the standard sequences of $S$, we see that
$$
n(\mu) = \sum_{s \in S} \left( \frac{\ell(s)(\ell(s) - 1)}{2} \right)\,,
$$
where $\mu$ is the weight of $T$.  Applying Theorem \ref{JMequation} to each standard sequence $s \in S$ finishes the proof.
\end{proof}

We have just shown that Morse's non-negative reformulation of the $\kcocharge$ is related to a generalization of Definition \ref{chargeindexvector}.  Furthermore, it is the work of \cite{MR3115329} which relates Definition \ref{LPstdcch} of the $\kcocharge$ to Definition \ref{stdLch} of the $\kcharge$ of a semi-standard $\ktableau$.

\begin{theorem}
\cite{MR3115329}
\label{sschcchrelL}
Given a semi-standard $\ktableau$ $T$ of weight $\mu$ and shape $\lambda$, 
$$
\kcharge(T) = n(\mu) - |Int^k(\lambda)| - \kcocharge(T).
$$
\end{theorem}

Finally, we can state a non-negative reformulation of the $\kcharge$ of a semi-standard $\ktableau$ by applying Corollary \ref{sscochargeLM} and Theorem's \ref{JMequationssk} and \ref{sschcchrelL}.

\begin{corollary}
\label{mainresult}
Let $T$ be a semi-standard $\ktableau$.  If $S$ is the set of all standard sequences of $T$, then the non-negative
$$
\kcharge(T) = \sum_{s \in S} \sum_{i=1}^{\ell(s)} \left( J_i^{(s)} + diag(\iup, c^{(i)}) \right)\,,
$$
where $\iup$ is the highest occurrence of $i$ in $s$ and $c^{(i)}$ is the highest addable cell of $T_{\leq i_r}$.
\end{corollary}

\begin{example}
For $k=4$, recall the semi-standard $\ktableau$ from Example \ref{stdseqchoiceex} is
$$
T = {\text{\footnotesize{\tableau[scY]{\tf 7_0 \cr \tf 6_1 \cr \tf 5_2 & 6_3 \cr \tf 3_3 & 4_4 & \tf 7_0 \cr \tf 2_4 & 3_0 & 5_1 & \tf 5_2 & 6_3 \cr 1_0 & \tf 1_1 & 2_2 & \tf 3_3 & 4_4 & \tf 4_0 & 5_1 & \tf 5_2 & 6_3}}}} 
$$
where the bold cells show the first standard sequence of cells.  Table \ref{tab:semistdchargeextable} gives $\kcharge(T) = 12.$
\end{example}

\begin{table}[h]
\renewcommand{\arraystretch}{1.4}
\begin{center}
\footnotesize{
\begin{tabular}{| c | c | c | c | c | c |}
\hline
$i_r$ & diag$({i}^{\uparrow}, {(i-1)}^{\uparrow})$ & $I_i$ & High $T_{\leq i_r}$-residue order & $J_i$ & diag$({i}^{\uparrow},c^{(i)})$ \\
\hline 
$1_1$ & - & $0$ & - & $0$ & 0 \\
\hline
$2_4$ & 0 & $0 - 0 = 0$ & $3 > 2 > 1 > 0 > 4$ & $0$ & $0$  \\
\hline
$3_3$ & 0 & $0 - 0 = 0$ & $2 > 1 > 0 > 4 > 3$ & $0$ & $0$ \\
\hline
$4_0$ & 1 & $0 + 1 + 1 = 2$ & $2 > 1 > 0 > 4 > 3$ & $1$ & $1$ \\
\hline
$5_2$ & 0 & $2 - 1 = 1$ & $1 > 0 > 4 > 3 > 2$ & $1$ & $0$ \\
\hline
$6_1$ & 0 & $1 - 0 = 1$ & $0 > 4 > 3 > 2 > 1$ & $1$ & $0$ \\
\hline
$7_0$ & 0 & $1 - 0 = 1$ & $4 > 3 > 2 > 1 > 0$ & $1$ & $0$ \\
\hline\hline
$1_0$ & - & $0$ & - & $0$ & 0 \\
\hline
$2_2$ & 0 & $0 + 1 + 0 = 1$ & $4 > 3 > 2 > 1 > 0$ & $1$ & $0$ \\
\hline
$3_0$ & 0 & $1 - 0 = 1$ & $3 > 2 > 1 > 0 > 4$ & $1$ & $0$ \\
\hline
$4_4$ & 0 & $1 - 0 = 1$ & $2 > 1 > 0 > 4 > 3$ & $1$ & $0$ \\
\hline
$5_1$ & 0 & $1 + 1 + 0 = 2$ & $2 > 1 > 0 > 4 > 3$ & $2$ & $0$ \\
\hline
$6_3$ & 0 & $2 - 0 = 2$ & $1 > 0 > 4 > 3 > 2$ & $2$ & $0$ \\
\hline
\end{tabular}
}
\end{center}
\caption{$\kcharge$ of $T$ from Example \ref{stdseqchoiceex}} 
\label{tab:semistdchargeextable}
\end{table}

\section{Acknowledgements}
\label{sec:ack}
\vspace{-0.05in}
The author would like to thank L. Lapointe, J. Morse, A. Schilling and M. Zabrocki for inspirational conversations through out the research and writing of this paper.
\vspace{-0.1in}

\bibliographystyle{alpha}
\bibliography{avi}
\label{sec:biblio}
\Addresses
\end{document}